\documentclass[11pt]{article}

\usepackage{amsmath,amssymb,amsthm,graphicx}
\usepackage{a4wide}
\usepackage{showlabels}
\usepackage{caption}
\usepackage{subcaption}
\usepackage{tikz}
\usepackage{authblk}
\usepackage{enumerate}
\usepackage{cite}

\usepackage{color}

\newcommand{\R}{\mathbb{R}}
\newcommand{\N}{\mathbb{N}}

\newcommand{\sg}[1]{\textnormal{sign}\left(#1\right)}

\newcommand{\cut}{\textnormal{cut}}

\newtheorem{lemma}{Lemma}

\newtheorem{ques}[lemma]{Question}
\newtheorem{conj}[lemma]{Conjecture}
\newtheorem{theorem}{Theorem}

\newtheorem{definition}{Definition}

\numberwithin{lemma}{section}
\numberwithin{fact}{section}
\numberwithin{equation}{section}

\title{The $p$-spectral radius of the Laplacian}

\author[1]{Elizandro Max Borba\thanks{elizandro.max@ufrgs.br}}
\author[2]{Sebastian Richter\thanks{sebastian.richter@mathematik.tu-chemnitz.de}}
\author[3]{Eliseu Fritscher\thanks{eliseu.fritscher@gmail.com}}
\author[4]{Carlos Hoppen\thanks{choppen@ufrgs.br}}

\affil[1,3,4]{Instituto de Matem\'atica, Universidade Federal do Rio Grande do Sul, Brazil}
\affil[1]{\'Area de Ci\^encias Exatas, Universidade Estadual do Rio Grande do Sul, Brazil}

\affil[2]{Fakult\"at f\"ur Mathematik, Technische Universit\"at Chemnitz, Germany}

\begin{document} \maketitle

\begin{abstract} The $p$-spectral radius of a graph $G=(V,E)$ with adjacency matrix $A$ is defined as $\lambda^{(p)}(G)=\max \{x^TAx : \|x\|_p=1 \}$. This parameter shows remarkable connections with graph invariants, and has been used to generalize some extremal problems. In this work, we extend this approach to the Laplacian matrix $L$, and define the $p$-spectral radius of the Laplacian as $\mu^{(p)}(G)=\max \{x^TLx : \|x\|_p=1 \}$. We show that $\mu^{(p)}(G)$ relates to invariants such as maximum degree and size of a maximum cut. We also show properties of $\mu^{(p)}(G)$ as a function of $p$, and a upper bound on $\max_{G \colon |V(G)|=n} \mu^{(p)}(G)$ in terms of $n=|V|$ for $p\ge 2$, which is attained if $n$ is even.

\noindent \textbf{Keywords:} Laplacian Matrix, p-spectral radius


\end{abstract}

\section{Introduction and main results}

Let $G=(V,E)$ be a simple $n$-vertex graph at least one edge with adjacency matrix $A$ and Laplacian matrix $L$. We recall that $L=D-A$, where $D$ is the diagonal matrix of vertex degrees. 

It is well known that obtaining the least and the largest eigenvalues ($\lambda_1$ and $\lambda_n$, respectively) of a real symmetric matrix $M\in\R^{n\times n}$ can be viewed as an optimization problem using the Rayleigh-Ritz Theorem \cite[Theorem 4.2.2]{HornJo}:
$$\lambda_1(M)=\min_{\|x\|=1}x^TMx\le\frac{x^TMx}{x^Tx}\le\max_{\|x\|=1}x^TMx=\lambda_n,$$
where $x\in\R^n$. Using the fact that $x^TAx=2\sum_{ij\in E} x_ix_j$, Keevash, Lenz and Mubayi \cite{KLM14} replaced the Euclidean norm $\|x\|$ by the $p$-norm $\|x\|_p$, where $p \in [1,\infty]$, and defined the \emph{$p$-spectral radius} $\lambda^{(p)}(G)$:

$$\lambda_p(G)=\max_{\|x\|_p=1}2\sum_{ij\in E} x_ix_j.$$


This parameter shows remarkable connections with some graph invariants. For instance, $\lambda^{(1)}(G)$ is equal to the Lagrangian $\mathfrak{L}_G$ of $G$, which was defined by Motzkin and Straus \cite{MotStr65} and satisfies $2\mathfrak{L}_G-1=1/\omega(G)$, where $\omega(G)$ is the clique number of $G$. Obviously $\lambda^{(2)}(G)$ is the usual spectral radius, and it can be shown that $\lambda^{(\infty)}(G)/2$ is equal to the number of edges of $G$.

An interesting result involving this parameter is about \emph{$K_r$-free graphs}, that is, graphs that do not contain a complete graph with $r$ vertices as a subgraph. Tur\'an \cite{Tur41} proved that, for all positive integers $n$ and $r$, the balanced complete $r$-partite graph, known as a \emph{Tur\'an graph} $T_r(n)$, is the only graph with maximum number of edges among all $K_{r+1}$-free graphs of order $n$. Kang and Nikiforov  \cite{KangNiki14} proved that, for $p\ge 1$, the graph $T_r(n)$ is also the only graph that maximizes $\lambda^{(p)}(G)$ over $K_{r+1}$-free graphs of order $n$, thus generalizing Tur\'an's result (which is the case $p=\infty$). Other results were obtained and extended to hypergraphs \cite{Niki14}.

This motivates us to extend this approach to the Laplacian matrix $L$, replacing the Euclidean norm by the $p$-norm. As $x^TLx=\sum_{ij\in E} (x_i-x_j)^2$, we define the $p$-spectral radius of the Laplacian as follows:

\begin{definition}
Let $G=(V,E)$. The $p$-spectral radius of the Laplacian matrix of $G$ is given by
$$\mu^{(p)}(G)=\max_{\|x\|_p=1} \sum_{ij\in E} (x_i-x_j)^{2}.$$ 
\end{definition}

According to Mohar \cite{mohar1991}, the Laplacian matrix is considered to be more natural than the adjancency matrix. It is a discrete analog of the Laplace operator, which is present in many important differential equations. The Kirchhoff Matrix-Tree theorem is a early example of the use of $L$ in Graph Theory. The largest eigenvalue (spectral radius) of $L$ has been associated, for example, with degree sequences of a graph \cite{GutVidSte2002,LSC09,AndMor1985,Pan2002}. The second smallest eigenvalue and its associated eigenvectors have also been studied since the seminal work by Fiedler \cite{Fiedler73}, which has been used in graph partitioning and has led to an extensive literature in spectral clustering. For more information about this area, see the survey~\cite{vonLux} and the references therein.

Therefore we hope that the definition of $\mu^{(p)}$ will shed some light on classical parameters of graph theory. In fact, we show that, in the same fashion as $\lambda^{(p)}(G)$, the parameter $\mu^{(p)}(G)$ relates to graph invariants, such as the maximum degree and the size of a maximum cut. We also show some properties of $\mu^{(p)}(G)$ as a function of $p$. The main results are:
\begin{theorem}\label{t:main}
Let $G=(V,E)$ be a graph with at least one edge. Then

\begin{enumerate}[(a)]
  \item $\mu^{(1)}(G)$ is equal to the maximum degree of $G$;
  \item $\mu^{(\infty)}(G)/4$ is equal to the size of a maximum cut  of $G$.
  \item The function $f_G : [1, \infty) \to \R$ defined by $f_G(p)=\mu^{(p)}(G)$ is strictly increasing, continuous and converges when $p\to\infty$;
\end{enumerate}
\end{theorem}

It seems to be the case that, by varying $p$, the vector $x$ that achieves $\mu^{(p)}(G)$ defines a maximum cut of the graph under different restrictions. For instance, $\mu^{(1)}(G)$ 
leads to a maximum cut with the constraint that one of the classes is a singleton, while $\mu^{(\infty)}(G)$ is gives a maximum cut with no additional constraint.
A rigorous basis for this statement remains a question for further investigation.

From the computational complexity point of view, it is interesting to note that computing $\mu^{(1)}(G)$ is easy (can be done in linear time), while computing $\mu^{(\infty)}(G)$ is an NP-complete problem, it is equivalent to finding the size of a maximum cut of $G$. For $\lambda^{(p)}$, the opposite happens: finding $\lambda^{(1)}(G)$ is NP-complete (equivalent to finding the clique number of $G$), while $\lambda^{(\infty)}(G)$ can be found in linear time.

We also present an upper bound on $\mu^{(p)}(G)$ if $p\ge 2$, which is attained for even $n$.
\begin{theorem}\label{t:mu_subg_bip}
Let $G=(V,E)$ be a graph with $n=|V|$. Then for $p\ge 2$,
\[ \mu^{(p)}(G)\le n^{2-2/p}. \]
If $n$ is even, equality holds if and only if  $G$ contains $K_{n/2,n/2}$ as subgraph.
\end{theorem}

Note that this means that, for even $n$, the value of $\mu^{(p)}(K_n)$ is the same as the value for the balanced complete bipartite graph with $n$ vertices. We conjecture that this holds for all $n$. 

This paper is organized as follows. In the remainder of the section we introduce some notation. In sections \ref{s:t1} and \ref{s:t2} we prove Theorems \ref{t:main} and \ref{t:mu_subg_bip}, respectively. In section \ref{s:conc} we present some additional remarks, conjectures and questions for future research.

Before proving our results, we set the notation used throughout the paper. The objective function of the optimization problems is 
\[F_G(x)=x^TLx=\sum_{ij\in E(G)} (x_i-x_j)^2.\]
We may drop the subscript of $F_G$ if $G$ is clear from context. It can be readily seen that $F_{G'}(x)\le F_G(x)$ for a subgraph $G'$ of $G$, and so $F_G(x)\le F_{K_n}(x)$ for any $n$-vertex graph $G$. Furthermore, $F_G(x)=0$ if $x$ is constant in each connected component of $G$. 

Finally, given an $n$-vertex graph $G=(V,E)$ and a vector $x \in \mathbb{R}^n$, the vertex sets $P,N$ and $Z$ are those on which $x_i$ is positive, negative, or equal to zero, respectively. We write $d_i$ for the degree of vertex $i$, and $d_{ij}$ is the number of edges between vertices $i$ and $j$ (0 or 1). The all-ones vector in $\mathbb{R}^n$ is $e$ and the $i$-th vector of the canonical basis of $\mathbb{R}^n$ is $e_i$.

\section{Proof of Theorem \ref{t:main}}\label{s:t1}

In this section, we prove Theorem \ref{t:main}, which relates $\mu^{(p)}(G)$ relates to graph invariants and gives properties of $\mu^{(p)}(G)$ as a function of $p$. Item (a) states that $\mu^{(1)}(G)$ is equal to the maximum degree of $G$. In order to prove it, 
 we need two lemmas.
\begin{lemma}\label{mu1a}
Let $x\in\R^n$ such that $\|x\|_1=1$ and $F_G(x)=\mu^{(1)}(G)$. Then at most one entry of $x$ or $-x$ is positive.
\end{lemma}
\begin{proof}
Let $x$ be as above. Without loss of generality, suppose $a,b\in P$ and define $x'$ and $x''$ as
\begin{eqnarray*}
  x'_k=
\begin{cases}
    x_a+x_b & \mbox{ if } k=a;\\
    0 &       \mbox{ if } k=b;\\
    x_k &     \mbox{ otherwise. }
\end{cases} & \mbox{ and } &
x''_k=
\begin{cases}
    0 &       \mbox{ if } k=a;\\
    x_a+x_b & \mbox{ if } k=b;\\
    x_k &     \mbox{ otherwise. }
\end{cases}
\end{eqnarray*}
Consider the differences $\Delta'=F(x')-F(x)$ e $\Delta''=F(x'')-F(x)$.
\begin{eqnarray*}
\Delta'=(d_a-d_{ab})(2x_ax_b+x_b^2)-2x_b\sum_{aj\in E,j\ne b} x_j
       -d_bx_b^2+2x_b\sum_{bj\in E,j\ne a} x_j
        + 4d_{ab}x_ax_b
\end{eqnarray*}
The expression for $\Delta''$ can be readily obtained switching the roles of $a$ and $b$. As $x_a,x_b>0$ we can take
\[
\frac{\Delta'}{x_b}+\frac{\Delta''}{x_a}=(d_a+d_b+d_{ab})(x_a+x_b)>0,
\]
so that at least one of the differences $\Delta'$ and $\Delta''$ is positive. This contradicts the maximality of $x$.
\end{proof}

So we can assume that $|P|=|N|=1$.

\begin{lemma}\label{mu1b}
Let $x\in\R^n$ such that $\|x\|_1=1$, $P=\{a\}$, $N=\{b\}$ and $d_a\ge d_b$. Then $d_a=F(e_a)\ge F(x)$, with equality if and only if $d_a=d_b=d_{ab}$.
\end{lemma}
\begin{proof}
Note that $x_a^2+x_b^2<1$, because $|x_a|+|x_b|=1$. Then
\begin{eqnarray*}
F(x)&=&  d_ax_a^2+d_bx_b^2+d_{ab}(1-x_a^2-x_b^2) \\
      & \le & d_a(x_a^2+x_b^2)+d_{ab}(1-x_a^2-x_b^2) \le d_a=F(e_a).
\end{eqnarray*}
The first and second inequalities become equalities if and only if $d_a=d_b$ and $d_a=d_{ab}$, respectively.
\end{proof}

So $\mu^{(1)}(G)$ is obtained for a vector $e_a$ for a vertex $a$ with maximum degree. That proves item (a) of Theorem \ref{t:main}. Note that the solutions are always of this form if the maximum degree is at least 2, because the equality situation of Lemma \ref{mu1b} is of interest only if the maximum degree is one. For instance, for $G=K_2$, any feasible vector attains the maximum.

Now we proceed to prove item (b), which states that $\mu^{(\infty)}(G)/4$ is equal to the size of a maximum cut  of $G$. 
In this case, the problem is of the form
\[ \mu^{(\infty)}(G) =\max_{\max_i |x_i|=1}\sum_{ij\in E} (x_i-x_j)^{2}. \]

\begin{lemma}\label{muinf}
Let $x\in\R^n$ such that $\max_i |x_i|=1$ and $F_G(x)=\mu^{(\infty)}(G)$. Then $|x_i|=1$, for all $i\in V$.
\end{lemma}
\begin{proof}
Let $x$ be as stated above. Suppose that there is $a\in V$ with $-1<x_a<1$. Define $x',x''\in\mathbb{R}^n$ as
\begin{eqnarray*}
  x'_i=
\begin{cases}
    1   & \mbox{ if } i=a;\\
    x_i & \mbox{ otherwise. }
\end{cases} & \mbox{ and } &
x''_i=
\begin{cases}
    -1  & \mbox{ if } i=a;\\
    x_i & \mbox{ otherwise. }
\end{cases}
\end{eqnarray*}
Consider the differences $\Delta'=F(x')-F(x)$ and $\Delta''=F(x'')-F(x)$. Then
\begin{eqnarray*}
\Delta'=d_a(1-x_a^2)-2(1-x_a)\sum_{aj\in E} x_j
\end{eqnarray*}
and similarly
\begin{eqnarray*}
\Delta''=d_a(1-x_a^2)+2(1+x_a)\sum_{aj\in E} x_j,
\end{eqnarray*}
and therefore 
$$\frac{\Delta'}{1-x_a}+\frac{\Delta''}{1+x_a}=2d_a>0.$$
So at least one of the differences $\Delta'$ and $\Delta''$ is positive. This contradicts the maximality of $x$.
\end{proof}

Now for a vector $x$ in the form given by Lemma \ref{muinf} let $S=\{i\in V : x_i=1\}$ and $T=\{i\in V : x_i=-1\}$. So
\begin{eqnarray*}
F(x)=\sum\limits_{i\in S, j\in T}(x_i-x_j)^2=4\cut(S,T).
\end{eqnarray*}
Then of course $F_G(x)=\mu^{(\infty)}(G)$ if $\cut(S,T)$ is a maximum cut. That proves item (a) of Theorem \ref{t:main}. Also, the maximum among graphs of order $n$ is
\begin{eqnarray*}
  \mu^{(\infty)}(K_n)=\mu^{(\infty)}(K_{\lfloor n/2 \rfloor,\lceil n/2 \rceil})=
\begin{cases}
   n^2     & \mbox{ if $n$ is even} ;\\
   n^2-1   & \mbox{ if $n$ is odd} .
\end{cases}
\end{eqnarray*}

Finally we prove item (c), which shows properties of the function $f_G : [1, \infty) \to \R$ defined by $f_G(p)=\mu^{(p)}(G)$. Namely, the function is strictly increasing (Lemma \ref{mu_cresc}), continuous (Lemma \ref{mu_cont}) and converges when $p\to\infty$ (Lemma \ref{mu_limite}). 
First we state two technical lemmas that will be useful.

\begin{lemma}\label{pnorm_bound}
Let $q\ge p\ge 1$. Then for $x\in\mathbb{R}^n$,
\[\|x\|_{q}\le \|x\|_p\le n^{\frac{1}{p}-\frac{1}{q}}\|x\|_{q}.\]
Furthermore. $|x_i^*|=n^{-1/q}\|x\|_{q}$ holds for a nonzero vector $x^*$ that attains the upper bound.
\end{lemma}

\begin{proof}
Without loss of generality, we can consider that $x$ has positive entries and $\|x\|_{q}=1$. The lower bound holds because the $p$-norm is decreasing on $p$, and it is attained by $e_i$. By applying the power mean inequality to the entries of $nx$, we can see that the upper bound is attained if and only if all entries are $n^{-1/q}$.
\end{proof}

\begin{lemma}\label{mu2p_lbound}
Let $G=(V,E)$ be a graph and $x\in R^n$ with $\|x\|_p=1$ and $p\ge2$. Then $F_G(x)\le n^{1-2/p}\mu^{(2)}(G)$.
\end{lemma}
\begin{proof}
By Rayleigh-Ritz theorem, we have $F_G(x)\le \|x\|_2^2\mu^{(2)}(G)$ for $x\ne 0\in \mathbb{R}^n$.
Using Lemma \ref{pnorm_bound}, we obtain $F_G(x)\le \|x^*\|_p^2\mu^{(2)}(G)=n^{1-2/p}\mu(G)$.
\end{proof}

The proof will be broken down in three lemmas, one for each result.

\begin{lemma}\label{mu_cresc}
For a graph $G$ and $p\ge 1$, $\mu^{(p)}(G)$ is strictly increasing in $p$.
\end{lemma}
\begin{proof}
Let $x\in\R^n$ such that $\|x\|_p=1$ and $F(x)=\mu^{(p)}(G)$, and $p'>p> 1$. Define $x':=x/\|x\|_{p'}$. As $\|x\|_{p'}\le 1$, we have
\begin{equation}\label{inc_quota}
\mu^{(p')}(G)\ge F(x')=\frac{1}{\|x\|_{p'}^2}F(x)\ge\mu^{(p)}(G). 
\end{equation} 
As $G$ has at least one edge $ij$, $\mu^{(p)}(G)>0$; pick $x$ such that $x_i=-x_j=2^{-1/p}$, and $x_i=0$ otherwise. Equality holds in equation \ref{inc_quota} if and only if $x=e_i$ for some $i$. We argue now that for $p>1$, $e_i$ never attains the maximum, so that $\mu^{(p)}(G)$ is strictly increasing.

For $p>1$, the stationarity conditions of the problem are $Lx=\lambda\nabla_x (|x_1|^p+\cdots+|x_n|^p-1)$. Note that $x\to|x|^p$ is differentiable for $p>1$. The $j$-th equation is
\begin{equation}\label{lagrange_j}
 d_jx_j-\sum_{jk\in E}x_k=
\begin{cases}
p|x_j|^{p-1}\sg{x_j},  & \mbox{if } x_j\ne 0;\\
0, & \mbox{if } x_j=0.
\end{cases}
\end{equation}

Without loss of generality, assume $G$ has no isolated vertices (as they don't contribute to the sum in $F$). Let $i\in V$ and $j$ a neighbor of $i$. Taking $x=e_i$, then $x_k=0$ if $k\ne i$; in particular, $x_j=0$. Then the right hand side of (\ref{lagrange_j}) is 0, and the left hand side is $d_jx_j-\sum_{jk\in E}x_k=0-x_i=-1$. Therefore, $e_i$ doesn't satisfy the optimality conditions of the problem, that is, for any $i\in V$, $F(e_i)<\mu^{(p)}(G)$ for $p>1$.

With this last statement in mind, recall that, by the proof of item a of Theorem \ref{t:main}, $\mu^{(1)}(G)=F(e_i)$ for $i$ with maximum degree. Therefore, we conclude that $\mu^{(1)}(G)<\mu^{(p)}(G)$ for $p>1$. This completes the proof.
 
\end{proof}

\begin{lemma}\label{mu_cont}
For any graph $G$ and $p\ge 1$, the function $p\to\mu^{(p)}(G)$ is continuous.
\end{lemma}
\begin{proof}
Let $x'\in\R^n$ such that $\|x'\|_{p'}=1$ and $F(x')=\mu^{(p')}(G)$, and $p'>p\ge 1$. By Lemma \ref{pnorm_bound} that $\|x'\|_p\le n^{\frac{1}{p}-\frac{1}{p'}}\|x'\|_{p'}$. Define $x:=x'/\|x'\|_{p}$. Then
\[ \mu^{(p')}(G)=F(x')=\|x'\|_p^2 F(x)\le\ n^{\frac{2}{p}-\frac{2}{p'}}\|x'\|_{p'}\mu^{(p)}(G)=n^{\frac{2}{p}-\frac{2}{p'}}\mu^{(p)}(G) \]
By Lemma \ref{mu_cresc}, we know that $\mu^{(p')}(G)>\mu^{(p)}(G)>0$. We also know from spectral graph theory (check for example \cite{Cve79}) that $\mu^{(2)}(G)\le\mu^{(2)}(K_n)=n$. Combining this with Lemma \ref{mu2p_lbound}, we have $\mu^{(p)}(G)\le n^{2-2/p}$ para $p\ge 2$; as $\mu^{(p)}(G)$ is strictly increasing in $p$ (Lemma \ref{mu_cresc}), this bound holds for $p\ge 1$. So

\begin{eqnarray*}
  \mu^{(p')}(G)-\mu^{(p)}(G)
  &\le& n^{\frac{2}{p}-\frac{2}{p'}}\mu^{(p)}(G)-\mu^{(p)}(G) \\
  &\le& \left(n^{\frac{2}{p}-\frac{2}{p'}}-1\right) n^{2-2/p}\\
  &<  & (n^{2(p'-p)}-1)n^2.
\end{eqnarray*}
So we have $\mu^{(p')}(G)-\mu^{(p)}(G)<\epsilon$ if $p'-p<\frac{1}{2}\log_n(\epsilon/n^2+1)$.
\end{proof}

\begin{lemma}\label{mu_limite}
For any graph $G$,
\[ \lim_{p\rightarrow\infty} \mu^{(p)}(G)=\mu^{(\infty)}(G). \]
\end{lemma}
\begin{proof}
For a given $p$, let $x$ such that $\|x\|_p=1$ and $F(x)=\mu^{(p)}(G)$. By the proof of Lemma \ref{mu_cresc}, we know that $x\ne e_i$, so $\max_i|x_i|<1$. Define $x':=x/\max|x_i|$. 
We can choose $N=N(x')\in\N$ such that 

\[ \mu^{(p)}(G)=F(x)=(\max|x_i|)^2F(x')>(\max|x_i|)^N\mu^{(\infty)}(G), \]
so that $0<\mu^{(\infty)}(G)-\mu^{(p)}(G)<(1-(\max|x_i|)^N)\mu^{(\infty)}(G).$ One can check that $\max|x_i|\ge n^{-1/p}$. 
The proof concludes noting that
\[ 0<\mu^{(\infty)}(G)-\mu^{(p)}(G)<(1-n^{-N/p})\mu^{(\infty)}(G), \]
and $n^{-N/p}\rightarrow 1$ when $p\rightarrow\infty$.

\end{proof}

\section{Proof of Theorem \ref{t:mu_subg_bip}}\label{s:t2}

In this section we prove Theorem \ref{t:mu_subg_bip}, which establishes the upper bound $\mu^{(p)}(G)\le n^{2-2/p}$ for $p\ge 2$, as well as a necessary and sufficient condition for equality. We denote $G=(S,T,E)$ a bipartite graph with vertex classes $S$ and $T$. First we state three auxiliary lemmas.
\begin{lemma}\label{mu1_same_sign}
Let $G=(S,T,E)$ be a bipartite graph, and $x\in\R^n $ such that $\|x\|_p=1$ and $F(x)=\mu^{(p)}(G)$. Then for $x$ or $-x$ we have $P=S$ and $N=T$.
\end{lemma}
\begin{proof}
Let $x$ be as stated above. Note that we can freely invert the entry signs preserving feasibility. Without loss of generality, if we invert the signs of negative entries in $S$ and positive entries in $T$, we are replacing, in the sum of $F$, terms of the form $(|x_i|-|x_j|)^2$ by $(|x_i|+|x_j|)^2$, thus increasing $F$.
\end{proof}

\begin{lemma}\label{mu2_bip_same_value}
Let $G=(S,T,E)$ be a bipartite graph, and $x\in\R^n$ such that $\|x\|_p=1$ and $F(x)=\mu^{(p)}(G)$. Then for $p\ge 2$, if $i$ and $j$ are in the same class, then $x_i=x_j$.
\end{lemma}
\begin{proof}
Suppose $x$ as stated above has entries with $i,j\in S$($=P$ without loss of generality, by Lemma \ref{mu1_same_sign}) with $x_i\ne x_j$. So
\[ F(x)=\sum_{j\in T}\sum_{ij\in E}(x_i-x_j)^2. \]
Let $M_p$ denote the power mean of $\{x_i : i\in S\}$. We exchange each $x_i$ by $M_p$. One can check that feasibility is preserved.
For fixed $j$, it is sufficient to check the variation of $\sum_i x_i^2+2\sum_i x_ix_j$:
\[  |S|M^2_p+2|S|M_p x_j>|S|M^2_2+2|S|M_1 x_j=\sum_i x_i^2+2\sum_i x_ix_j. \]
The inequality holds by the power mean inequality. So the exchange increases $F$, contradicting the maximality of $x$.
\end{proof}

This allows us to obtain a formula for complete bipartite graphs.

\begin{lemma}\label{p:mu_bip_comp}
Let $G=(S,T,E)$ be a complete bipartite graph. For $p\ge 2$,
\[  \mu^{(p)}(G)=|S||T|(a+b)^2, \]
where
\begin{eqnarray*}
  a=\left(|S|+|T|\left(\frac{|S|}{|T|}\right)^{\frac{p}{p-1}}\right)^{-1/p}, & b=\left(\dfrac{|S|}{|T|}\right)^{\frac{1}{p-1}}a.
\end{eqnarray*}
\end{lemma}
\begin{proof}
By Lemma \ref{mu2_bip_same_value}, we can assume $x_i=a$ for $i\in S$ and $x_i=-b$ for $i\in T$. Then apply Lagrange method to the function $g(a,b)=|S||T|(a+b)^2$ constrained by $h(a,b)=|S|a^p+|T|b^p=1$.
\end{proof}

In the proof of the item (c) of Theorem \ref{t:main}, the balanced complete bipartite graph attains the maximum for $\mu^{(\infty)}$ among graphs of order $n$. The same holds for $\mu^{(p)}$ if $2\le p<\infty$ if $n$ is even.

\begin{proof}[Proof of Theorem \ref{t:mu_subg_bip}]
As $\mu^{(2)}(K_n)=n$, the bound $\mu^{(p)}(G)\le n^{2-2/p}$ is a direct consequence of Lemma \ref{mu2p_lbound}.
By Lemma \ref{p:mu_bip_comp}, one can check that $\mu^{(p)}(K_{n/2,n/2})=n^{2-2/p}$. Furthermore, if $K_{n/2,n/2}\subseteq G$, the inequality is trivial, because $F_G(x)$ won't decrease if we add edges to $G$.

Now let $G$ and $x\in\R^n$ such that $F_G(x)=\mu^{(p)}(G)=n^{2-2/p}$. Note that $|x_i|=|x_j|, \forall i,j\in V$; otherwise, as $\mu^{(2)}(G)\le\mu^{(2)}(K_n)=n$ and by Lemma \ref{mu2p_lbound}, we would have $F_G(x)< n^{2-2/p}$. Also, $K_{|P|,|N|}=(P,N,E')$ is a subgraph of $G$; otherwise there would be $a\in P$ and $b\in N$ such that $\{a,b\}\notin E(G)$ and $F_{G\cup\{a,b\}}(x)>F_G(x)=n^{2-2/p}$, in contradiction with Lemma \ref{mu2p_lbound}.

Therefore, $F_{K_{|P|,|N|}}(x)=F_G(x)=n^{2-2/p}$, because the edges induced by $P$ or $N$ do not contribute to $F_G(x)$. 
Observe that, by Lemma \ref{p:mu_bip_comp}, $|x_i|=|x_j|$ if and only if $|P|=|N|$, therefore $|P|=|N|=n/2$.
\end{proof}

Although we conjecture that the equality condition of Theorem \ref{t:mu_subg_bip} also holds for odd $n$ (of course with a different quota given by \ref{p:mu_bip_comp}), the reasoning used in the proof does not work in this case, because then the balanced complete bipartite graph does not attain the bound given by Lemma \ref{mu2p_lbound}.

\section{Concluding remarks}\label{s:conc}


As already mentioned in the introduction, we seem to obtain maximum cuts under different restrictions in the graph by varying $p$. That motivates the following broad question for further investigation:

\begin{ques}
For $p\ge1$, which relation possibly exists between $\mu^{(p)}(G)$ and cuts (or other parameters) of $G$?
\end{ques}

Also, we proved that computing $\mu^{(1)}(G)$ can be done in linear time, while computing $\mu^{(\infty)}(G)$ is an NP-complete problem. As finding the maximum degree of $G$ can be trivially reduced in linear time to finding the size of a maximum cut of $G$, it might be the case that, by increasing $p$, we obtain a problem that is at least as hard. This motivates the following conjecture:

\begin{conj}
Let $q>p\ge1$. The problem of finding $\mu^{(p)}(G)$ can be reduced to the problem of finding $\mu^{(q)}(G)$ in polynomial time.
\end{conj}

There are other approaches that seek to generalize eigenvalues via the introduction of the $p$-norm. Amghibech \cite{Amghibech03} introduced a non-linear operator, which he called the $p$-Laplacian $\Delta_p$, that induces a functional of the form $\langle x,\Delta_p\rangle=\sum_{ij\in E} |x_i-x_j|^p$ instead of the quadratic form of the Laplacian. This functional is unbounded for $p=\infty$ over the $p$-norm unit ball, and the case $p=1$ cannot be treated directly. However, the eigenvalue formulation used allows to explore eigenvalues other than the largest and the smallest: $\lambda$ is said to be a $p$-eigenvalue of $M$ if there is a vector $v\in\R^n$ such that
$$(\Delta_p x)_i=\lambda\phi_p(v_i),\quad \phi_p(x)=|x|^{p-1}\sg{x}.$$
The vector $v$ is called a $p$-eigenvector of $M$ associated to $\lambda$. Using this formulation, B\"{u}hler and Hein \cite{BuhHei2009} proved that the cut obtained by ``thresholding'' (partitioning according to entries greater than a certain constant) an eigenvector associated to the second smallest eigenvalue of $\Delta_p$ converges to the optimal Cheeger cut  when $p\to 1$; in practice, the case $p=2$ is used to obtain an approximation to this cut \cite{ShiMal,vonLux}.

It may be possible to adapt this method to the standard Laplacian operator, which would allow us to explore a $p$-norm version of the second smallest eigenvalue of $L$, which could potentially also lead to different cuts according to the value of $p$.

\smallskip

\noindent \textbf{Acknowledgments} This work was partially supported by CAPES Grant PROBRAL 408/13 - Brazil and DAAD PROBRAL Grant 56267227 - Germany.

\end{document}